\documentclass[12pt]{amsart}
\usepackage{amssymb,amsmath,amsthm,comment}
\usepackage{pdfpages,graphicx} 
\usepackage[section]{placeins} % figures should be displayed in same section
\usepackage{wrapfig}
\usepackage{float}
\usepackage{lineno}
\usepackage{youngtab}

\usepackage{setspace}
\newcommand{\shrinkmargins}[1]{
  \addtolength{\textheight}{#1\topmargin}
  \addtolength{\textheight}{#1\topmargin}
  \addtolength{\textwidth}{#1\oddsidemargin}
  \addtolength{\textwidth}{#1\evensidemargin}
  \addtolength{\topmargin}{-#1\topmargin}
  \addtolength{\oddsidemargin}{-#1\oddsidemargin}
  \addtolength{\evensidemargin}{-#1\evensidemargin}
  }
\shrinkmargins{0.5}

\newtheorem{theorem}{Theorem}
\newtheorem{lemma}[theorem]{Lemma}

\newtheorem{corollary}[theorem]{Corollary}
\newtheorem*{theorem*}{Theorem}

{Claim}

\theoremstyle{definition}

\theoremstyle{remark}

\newtheorem{Challenge}{Challenge}
\numberwithin{theorem}{section} \numberwithin{equation}{section}

\usepackage{caption}
\usepackage{multirow}
\makeatletter
\@namedef{subjclassname@2020}{\textup{2020} Mathematics Subject Classification}
\makeatother

\def\nequiv{\mathrel{\rlap/\kern-.34em\equiv}}%
\def\func#1{\mathop{\rm #1}}%
\begin{document}
\title[Log-concavity]{Log-Concavity of Infinite Product and Infinite Sum Generating Functions}
\author{Bernhard Heim }
\address{Faculty of Mathematical and Natural Sciences, Mathematical Institute, University of Cologne, Weyertal 86--90, 50931 Cologne, Germany}
\email{bheim@math.uni-koeln.de}
\address{Lehrstuhl A f\"{u}r Mathematik, RWTH Aachen University, 52056 Aachen, Germany}
\email{bernhard.heim@rwth-aachen.de}
\author{Markus Neuhauser}
\address{Kutaisi International University, 5/7, Youth Avenue,  Kutaisi, 4600 Georgia}
\address{Lehrstuhl A f\"{u}r Mathematik, RWTH Aachen University, 52056 Aachen, Germany}
\email{markus.neuhauser@kiu.edu.ge}
%%\linenumbers
%%
%%
%%
\begin{abstract}
We
expand on the remark
by Andrews on the importance of infinite sums and products in combinatorics.
Let $\{g_d(n)\}_{d\geq 0,n \geq 1}$ be the double sequences $\sigma_d(n)= \sum_{\ell \mid
n} \ell^d$ or
$\psi_d(n)= n^d$. We
associate double sequences 
$\left\{ p^{g_{d}
}\left( n\right) \right\}$ and $\left\{ q^{g_{d}
}\left( n\right) \right\} $, defined as the coefficients of
\begin{eqnarray}
\sum_{n=0}^{\infty} p^{g_{d}
}\left( n\right) \,
t^{n} & := &  
\prod_{n=1}^{\infty} \left( 1 - t^{n} \right)^{-\frac{ \sum_{\ell \mid
n} \mu(\ell) \, g_d(n/\ell)   }{n} }, \label{ex}\\
\sum_{n=0}^{\infty}
q^{g_{d}
}\left( n\right) \, t^{n} & := & 
\frac{1}{1 - \sum_{n=1}^{\infty} g_d(n) \,
t^{n}
}. \label{ge}
\end{eqnarray}
These coefficients are related to the number of partitions $\mathrm{p}\left( n\right) = p^{\sigma _{1
}}\left ( n\right) $, plane partitions $\func{pp}\left( n\right) = p^{\sigma _{2
}}\left( n\right) $ of $n$, and Fibonacci numbers $F_{2n} =
q^{\psi _{1
}}\left( n\right) $. Let $n \geq 3$
and let $n \equiv 0 \pmod{3}$. Then the coefficients are log-concave at $n$ for almost all $d$ in the exponential (\ref{ex}) and geometric (\ref{ge}) cases. The coefficients are not log-concave for almost all $d$ in both cases, if $n \equiv 2 \pmod{3}$. Let $n\equiv 1 \pmod{3}$. Then the log-concave property flips for almost all $d$.
\end{abstract}
%%
%%
%%
%%%%%%%%%%%%%%%%%%%%%%%%%%%%%%%%%%%%%%%%%%%%%%%%%%%%%%%%%%%%%%%%%%%%%%%%%%%%%%%%%%%%%%%%%%%%%%%%%%%%%%%%%%
%%%%%%%%%%%%%%%%%%%%%%%%%%%%%%%%%%%%%%%%%%%%%%%%%%%%%%%%%%%%%%%%%%%%%%%%%%%%%%%%%%%%%%%%% Section 1
%%%%%%%%%%%%%%%%%%%%%%%%%%%%%%%%%%%%%%%%%%%%%%%%%%%%%%%%%%%%%%%%%%%%%%%%%%%%%%%%%%%%%%%%%%%%%%%%%%%%%%%%%%
%%%%%%%%%%%%%%%%%%%%%%%%%%%%%%%%%%%%%%%%%%%%%%%%%%%%%%%%%%%%%%%%%%%%%%%%%%%%%%%%%%%%%%%%%%%%%%%%%%%%%%%%%%
\maketitle
\newpage
\section{Introduction}
In this paper, we study
log-concave properties of families of sequences
related to infinite product and infinite sum generating functions \cite{Br89,St89,HN22a}.

Log-concavity is an important property.
For polynomials with positive coefficients, real-rootedness
entails log-concavity of all internal coefficients, which implies unimodality. 
Recent breakthrough works by Huh and his collaboraters, using methods in algebraic geometry, have proven the Mason and Heron--Rota--Welsh conjecture on the log-concavity of the chromatic polynomials of graphs, and finally the characteristic polynomials of matroids \cite{AHK18, BHMPW22, Hu12}. We refer
to the survey
by Kalai \cite{Ka22} on the work
by Huh. 
Note that
Zhang \cite{Zh22} proved that the coefficients of the Nekrasov--Okounkov polynomials are almost all unimodal, building on the work
by Odlyzko and 
Richmond \cite{OR85} and Hong and Zhang \cite{HZ21}.

We offer an approach for sequences associated with
generating functions, where in general, not all coefficients are log-concave. For example, it is well-known that the partition numbers $\mathrm{p}\left( n\right) $
are log-concave for $n \geq 26$.
We encounter $\mathrm{p}\left( n\right) $ and the number of plane partition numbers $\func{pp}\left( n\right) $ of $n$ \cite{An98, Kr16} and Fibonacci numbers $F_n$.
A sequence $\{a_n\}_n$ is
called log-concave at $n_{0}$, if
\begin{equation*} 
a_{n_0}^2 - a_{{n_0}-1} \, a_{{n_0}+1} \geq 0.
\end{equation*}

Let
$\{g_d(n)\}_{d\geq 0,n \geq 1}$ be a double sequence of positive integers.
We
examine the coefficients of the
associated generating functions of exponential (\ref{expo}) and geometric type (\ref{geom}):
\begin{eqnarray}
\sum_{n=0}^{\infty} p^{g_{d
}}\left( n\right) \,
t^{n} & := &  
 \exp    \left( \sum_{n=1}^{\infty} g_d(n) \, \frac{
t^{n}}{n} \right) = 
\prod_{n=1}^{\infty} \left( 1 - t
^{n} \right)^{-\frac{\alpha_d(n)}{n}},\label{expo} \\
\sum_{n=0}^{\infty}
q^{g_{d
}}\left( n\right) \,
t^{n} & := & 
\frac{1}{1 - \sum_{n=1}^{\infty} g_{d}\left( n\right) \,
t^{
n}}. \label{geom}
\end{eqnarray}
Here $\alpha_d(n) = \sum_{\ell \mid
n} \mu(\ell) \, g_d(n\, / \, \ell)$, where $\mu$ is the M\"{o}bius function.

The approach offered in this paper, is
incited by
Andrews' remark
(\cite{An98}, chapter 6, page 99) in the context of Meinardus'
theorem:
``Unfortunately not much is known about problems when a series rather than a product is involved".
We call
$n$ an exception related to a sequence $\left\{ a_{n}\right\} _{n}$, %$g_d$ in the exponential case
if
$$%\left(p^{g_{d}
%}\left(
a_{n}%\right) \right)
^2 - %p^{g_{d
%}}\left(
a_{n-1}%\right)
\,\, %p^{g_{d}
%}\left(
a_{n+1}%\right)
<0.$$
The set %s
of all exceptions %for fixed $d$ are
is denoted by 
$
{E}^{
%p^{g_{d}}
a}$.
%Similar we define pairs of  exceptions in the geometric case.

To this point only the exponential cases have
been studied in the literature. 
Let $g_d = \sigma _{d}
$. For fixed $d=1$, we have the number of partitions 
$\mathrm{p}\left( n\right) =p^{\sigma_1
}\left( n\right) $.

Nicolas \cite{Ni78} proved in 1978, that the partition function
$\mathrm{p}\left( n\right) $ is log-concave, if and only if
$n$ is not an element of the finite set 
$$
{E}^{
p^{\sigma _{1}}
}= \{
2k+1
\, \vert \, 0 \leq k \leq 12\}.$$
This was
proved again by DeSalvo and Pak \cite{DP15}. Both proofs utilize the
Rademacher formula for $\mathrm{p}\left( n\right) $.
In \cite{HNT22}, we have proven that the plane partition function 
$\func{pp}\left( n\right) = p^{\sigma_2
}\left( n\right) $ is log-concave for almost all $n$. Finally, based on numerical experiments, we conjectured that 
$$
{E}^{
p^{\sigma _{2
}}}= \{
2k+1
\, \vert \, 0 \leq k \leq 5\}.$$
Recently, the conjecture was proven by Ono, Pujahari, and Rolen \cite{OPR22}.

In this paper, we study the similarities between log-concavity properties 
of the coefficients obtained by the %exponential (\ref{expo}) and geometric
generating
function of exponential (\ref{expo}) and geometric
type (\ref{geom}).
\subsection{Landscape of Exceptions in the
Exponential Cases}
We consider log-concavity for $\left\{ p
^{g_{d}}\left( n\right) \right\} $. We recall the results obtained in
\cite{HN22a} and \cite{HN22b}. Note, the information on $d=0$
is new.
Numerical investigations indicate that 
\begin{equation*}\label{sigma0}
E^{p^{\sigma_0}}
=\left\{ 2k+1\middle| 0\leq k\leq 371\right\} \setminus
\left\{ 717,723,729,735,741\right\} 
\end{equation*}
tested up
to $n=25
00$. Further, for $0
\leq d \leq 5$ the cardinality of
$
{E}^{p^{\sigma _{d}
}}$ seems to be decreasing: $367 > 13 > 6 > 4> 2\geq 2$. But
$\left\vert
{E}^{p^{\sigma _{6}
}}\right\vert \geq 3 $. 
We refer to Table \ref{landscape1}. The case $g_d= \psi_d
$, if we see Table \ref{clean}, reveals
the similar pattern.

\begin{table}[H]                                  \[                                                
\begin{array}{r|cccccccccccccccccc}             \hline               
n\backslash d&0&{1}& {2}&3&4&5&6&7&8&9&10&11&12&13&14&15&16&17\\ \hline \hline  %%
1&\bullet&\bullet&\bullet&\bullet&\bullet&\bullet&\bullet&\bullet&\bullet&\bullet&\bullet&\bullet&\bullet&\bullet&\bullet&\bullet&\bullet&\bullet\\
2&&&&&&&&&&&&&&&&&&\\
3&\bullet&\bullet&\bullet&\bullet&&&&&&&&&&&&&&\\
4&&&&&&&\bullet&\bullet&\bullet&\bullet&\bullet&\bullet&\bullet&\bullet&\bullet&\bullet&\bullet&\bullet\\
5&\bullet&\bullet&\bullet&\bullet&\bullet&\bullet&\bullet&\bullet&\bullet&\bullet&&&&&&&&\\
6&&&&&&&&&&&&&&&&&&\\
7&\bullet&\bullet&\bullet&\bullet&&&&&&&&\bullet&\bullet&\bullet&\bullet&\bullet&\bullet&\bullet\\
8&&&&&&&&&&\bullet&\bullet&\bullet&\bullet&\bullet&\bullet&\bullet&\bullet&\bullet\\
9&\bullet&\bullet&\bullet&&&&&&&&&&&&&&&\\
10&&&&&&&&&&&&&&&&&\bullet&\bullet\\
11&\bullet&\bullet&\bullet&&&&&&&&&&\bullet&\bullet&\bullet&\bullet&\bullet&\bullet\\
12&&&&&&&&&&&&&&&&&&\\
13&\bullet&\bullet&&&&&&&&&&&&&&&&\\
14&&&&&&&&&&&&&&&&\bullet&\bullet&\bullet\\ \hline
\end{array}
\]
\caption{\label{landscape1}Exceptions for $p^{\sigma_d
}\left( n\right) $
for $0
\leq d
\leq 17
$  and $1
\leq n
\leq 14$.}
\end{table}
%%%%%%%%%%%
\begin{table}[H]
\[
\begin{array}{r|cccccccccccccccccccc}
\hline
n\backslash d&0&{1}&{2}&3&4&5&6&7&8&9&10&11&12&13&14&15&16&17\\ \hline \hline
1&&\bullet&\bullet&\bullet&\bullet&\bullet&\bullet&\bullet&\bullet&\bullet&\bullet&\bullet&\bullet&\bullet&\bullet&\bullet&\bullet&\bullet\\
2&&&&&&&&&&&&&&&&&&\\
3&&&&&&&&&&&&&&&&&&\\
4&&&&&&\bullet&\bullet&\bullet&\bullet&\bullet&\bullet&\bullet&\bullet&\bullet&\bullet&\bullet&\bullet&\bullet\\
5&&&&&&&&&&&&&&&&&&\\
6&&&&&&&&&&&&&&&&&&\\
7&&&&&&&&&&&&\bullet&\bullet&\bullet&\bullet&\bullet&\bullet&\bullet\\
8&&&&&&&&&&\bullet&\bullet&\bullet&\bullet&\bullet&\bullet&\bullet&\bullet&\bullet\\
9&&&&&&&&&&&&&&&&&&\\
10&&&&&&&&&&&&&&&&&\bullet&\bullet\\
11&&&&&&&&&&&&&\bullet&\bullet&\bullet&\bullet&\bullet&\bullet\\
12&&&&&&&&&&&&&&&&&&\\
13&&&&&&&&&&&&&&&&&&\\
14&&&&&&&&&&&&&&&&\bullet&\bullet&\bullet\\
\hline
\end{array}
\]
\caption{\label{clean}
Exceptions for $p^{\psi_{d
}}\left( n\right) $
for $0
\leq d
\leq 17
$ and             $
1\leq n
\leq 14
$.}
\end{table}
Now, fixing $n$ and studying log-concavity, reveals
a new phenomenon. Let
$n \geq 3$. Let $g_d(n)= \sigma_d(n)$
or $\psi_d(n)$. Then the set of all
exceptions for all $d$
is finite, if and only if $n \equiv 0 \pmod{3}$. More generally \cite{HN22b}, let
$\{g_d(n)\}_{d\geq 1,n \geq 1}$ be positive real numbers satisfying $g_d(1)=1$ and 
\begin{equation*}
0 \leq g_{d}\left( n\right) -n^{d} \leq 
\left( g_{1
} \left( n\right) -
n
\right)
\,  \left( n-1\right) ^{d-1}.
\label{eq:allgemein}
\end{equation*}
Let $n \geq 3$.
Then for almost all $d$, $p^{g_{d
}}\left( n\right) $ is log-concave
at $n$, if and only if $n$ is
divisible by~$3$. Moreover, explicit bounds are given.
It would be interesting to examine the results of this paper in the context of
generalized Laguerre-P\'olya functions and Jensen polynomials \cite{Wa22}.

\subsection{Landscape of Exceptions in the Geometric Cases}
At first glance, the {\it geometric} case, see Table \ref{landscapesigma} and Table \ref{landscape},
seems not to reveal much structure. Nevertheless, we recall that 
$q
^{\psi _{1
}}(n)= F_{2n}$
can be identified with the sequence of the $2n$th Fibonacci numbers, which is log-concave for $n>1$.
This follows from the fact that  $
q^{\psi_{1
}}(n)= U_{n-1}(\frac{3}{2})$, where $U_n(x)$ is the $n$th
Chebyshev polynomial of the second kind. Thus, we have some kind of analogue
to Nicolas' result.
Thus far, for $d=2$ and $\psi_2(n)$, we expect infinitely many exceptions.
Nevertheless, by fixing $n\geq 3$ we obtain the following new result.

We have the {\it geometric} cases for $g_d(n)= \sigma_d(n)$
in Table \ref{landscapesigma}
and $g_d(n)=n^d$ in Table~\ref{landscape}.
\begin{table}[H]                                  \[                                                \begin{array}{r|cccccccccccccccccc}
\hline
n\backslash d&0&
{1}& {2}&3&4&5&6&7&8&9&10&11&12&13&14&15&16&17\\ \hline \hline
1&\bullet&\bullet&\bullet&\bullet&\bullet&\bullet&\bullet&\bullet&\bullet&\bullet&\bullet&\bullet&\bullet&\bullet&\bullet&\bullet&\bullet&\bullet\\
2&&&&&&&&&&&&&&&&&&\\
3&\bullet&\bullet&\bullet&\bullet&\bullet&&&&&&&&&&&&&\\
4&&&&&&\bullet&\bullet&\bullet&\bullet&\bullet&\bullet&\bullet&\bullet&\bullet&\bullet&\bullet&\bullet&\bullet\\
5&\bullet&\bullet&\bullet&\bullet&\bullet&\bullet&&&&&&&&&&&&\\
6&&&&&&&&&&&&&&&&&&\\
7&\bullet&\bullet&\bullet&\bullet&\bullet&\bullet&\bullet&\bullet&\bullet&\bullet&\bullet&\bullet&\bullet&\bullet&\bullet&\bullet&\bullet&\bullet\\
8&&&&&&&&&&&&&&&&&&\\
9&\bullet&\bullet&\bullet&\bullet&\bullet&&&&&&&&&&&&&\\
10&&&&&&\bullet&\bullet&\bullet&\bullet&\bullet&\bullet&\bullet&\bullet&\bullet&\bullet&\bullet&\bullet&\bullet\\
11&\bullet&\bullet&\bullet&\bullet&\bullet&\bullet&\bullet&&&&&&&&&&&\\
12&&&&&&&&&&&&&&&&&&\\
13&\bullet&\bullet&\bullet&\bullet&\bullet&\bullet&\bullet&\bullet&\bullet&\bullet&\bullet&\bullet&\bullet&\bullet&\bullet&\bullet&\bullet&\bullet\\
14&&&&&&\bullet&\bullet&&&&&&&&&&&\\ \hline
\end{array}
\]
\caption{\label{landscapesigma}Exceptions for $
q^{\sigma_{d
}}\left( n\right) $
for $0\leq d\leq 17$
and $1\leq n\leq 14$.}
\end{table}
%%%%%%%%%%%%%%%%%%%%%%%%%%%%%%%%%%%%%%%%%%%
%%%%%%%%%%%%%%%%%%%%%%%%%% Q  n^d kommt nun.
\begin{table}[H]                                  
\[                                                
\begin{array}{r|cccccccccccccccccccc}             \hline     n\backslash d&0&{1}& {2}&3&4&5&6&7&8&9&10&11&12&13&14&15&16&17\\ \hline \hline
1&\bullet&\bullet&\bullet&\bullet&\bullet&\bullet&\bullet&\bullet&\bullet&\bullet&\bullet&\bullet&\bullet&\bullet&\bullet&\bullet&\bullet&\bullet\\
2&&&&&&&&&&&&&&&&&&\\
3&&&&&&&&&&&&&&&&&&\\
4&&&\bullet&\bullet&\bullet&\bullet&\bullet&\bullet&\bullet&\bullet&\bullet&\bullet&\bullet&\bullet&\bullet&\bullet&\bullet&\bullet\\
5&&&\bullet&\bullet&&&&&&&&&&&&&&\\
6&&&\bullet&&&&&&&&&&&&&&&\\
7&&&&&\bullet&\bullet&\bullet&\bullet&\bullet&\bullet&\bullet&\bullet&\bullet&\bullet&\bullet&\bullet&\bullet&\bullet\\
8&&&&\bullet&\bullet&\bullet&&&&&&&&&&&&\\
9&&&&\bullet&&&&&&&&&&&&&&\\
10&&&\bullet&&&\bullet&\bullet&\bullet&\bullet&\bullet&\bullet&\bullet&\bullet&\bullet&\bullet&\bullet&\bullet&\bullet\\
11&&&\bullet&&\bullet&\bullet&\bullet&&&&&&&&&&&\\
12&&&\bullet&\bullet&&&&&&&&&&&&&&\\
13&&&&\bullet&&\bullet&\bullet&\bullet&\bullet&\bullet&\bullet&\bullet&\bullet&\bullet&\bullet&\bullet&\bullet&\bullet\\
14&&&&&\bullet&\bullet&\bullet&&&&&&&&&&&\\ \hline
\end{array}
\]
\caption{\label{landscape}Exceptions for $
q^{\psi_{d
}}\left( n\right) $ for $0\leq d\leq 17$ and $1\leq n\leq 14$.}
\end{table}
\newpage
\subsection{Main
Results}
In this paper, we prove the following:
\begin{theorem}
\label{hauptsatzneu}
Let $\{g_d(n)\}_{d\geq 0,n\geq 1}$ be a double sequence of positive real numbers with $g_d(1)=1$ for all $d$
and
\begin{equation}\label{Bedingung}
0 \leq g_d(n) - n^d \leq \left( g_{
0}\left( n\right) -
1\right) \, (n-1)^{d
}.
\end{equation}
Suppose there is an
$0<r
\leq 1
$, such that $q
^{g_{0}
}\left( n\right) \leq r^{-n}$. Let $n\geq 3$ and $D^{g}\left( n\right) $ be defined
by
\begin{eqnarray}
D^{g}\left( n\right)
&:=
&-
2
\log
_{9/8}\left( r\right)
\,\, n,\label{eq:0mod3}\\
D^{g}\left( n\right)
&:=&
\log _{9/8}\left( 3\right)
-2n\log
_{9/8}\left( r\right)
-\log
_{9/8}\left(
n+2
\right)
\label{eq:1mod3}
\end{eqnarray}
for $n\equiv 0\pmod{3}$ in (\ref{eq:0mod3}) and
$n \equiv 1 \pmod{3}$ in (\ref{eq:1mod3}). 
Further, let $n \equiv 2 \pmod{3}$ and $n \neq 5$.
Then $D^{g}\left( n\right) $ is defined
as
\begin{equation*}
\log
_{9/8}\left( 2\right) -
\left( n+1\right)
\log
_{9/8}\left( r\right)
.
\end{equation*}
Moreover, $D^{g}\left( 5\right) :=
\log _{9/8}\left(
2q^{g_{0}
}\left( 4\right) q
^{g_{0}
}\left( 6\right)
\right)
$.

Let $d >
D^{g}\left( n\right) $.
Then
\begin{equation}
\frac{ \left(
q^{g_{d
}}\left( n\right) \right)^2}{
q^{g_{d
}}\left( n-1\right ) \,\, 
q^{g_{d
}}\left( n+1\right) }
< 1 \text{, if and only if } n \equiv 1 \pmod{3}.
\end{equation}
\end{theorem}

The double sequences $\{g_d(n)\}$ given by $\psi_d(n)= n^d$ and $\sigma_d(n)= \sum_{\ell \mid n} \ell^d$ satisfy (\ref{Bedingung}). 
In the case $\psi_{0}\left( n\right) $, we have 
$\frac{t
}{
1-t
}=\sum _{n=1}^{\infty }
t^{n}$ and
$\frac{1}{1-\frac{
t}{
1-t
}}=1+\sum _{n=1}^{\infty }
2^{n-1}
t^{n}
$.
Therefore, $
q^{\psi_{0
}}\left( n\right)
\leq 2^{n}$ for $n\geq 1$.
We can apply Theorem \ref{hauptsatzneu} with
$r=\frac{1}{2}$.

 Let 
\[
D^{\psi}\left( n\right) :=\left\{
\begin{array}{ll}
2
n\log
_{9/8}\left(
2\right)
,&n\equiv 0\pmod{3},\\
2
n\log
_{9/8}\left( 2\right) +
\log
_{9/8}\left( 3\right)
-\log
_{9/8}\left(
n+2
\right)
,&n\equiv 1\pmod{3},\\
\left( n+2
\right) \log
_{9/8}\left(
2\right)
,
&n\equiv 2\pmod{3},n\neq 5,\\
\log _{9/8}\left( 512
\right)
,&n=5.
\end{array}
\right.
\]

\begin{corollary} Let $n \geq 3$. Let $d >
D^{\psi}(n)$. Then
\begin{equation}
\frac{ \left(
q^{\psi _{d
}}\left( n\right)  \right)^2                      }{
q^{\psi _{d
}}\left( n-1\right)      \,\,
q^{\psi _{d
}}\left( n+1\right) }
< 1                 \text{, if and only if }n \equiv 1 \pmod{3}.
\end{equation}
\end{corollary}
For $g_{0
}=\sigma _{
0}$, we obtain
$\frac{1}{1-\sum _{n=1}^{\infty }\sigma _{
0}\left( n\right)
t^{n}}=1
 + t
 + 3\*
t^2
 + 7\*
t^3
 + 18\*
t^4
 + 43\*
t^5
 + 108\*t
^6+\ldots $.
Obviously, $\sigma _{
0}\left( n\right) \leq
n
$. Then
$\sum _{n=1
}^{\infty }
nt
^{n}=\frac{
t}{\left( 1-t
\right) ^{
2}}$ and
the  radius of convergence of the series expansion of
$\frac{1}{1-\frac{
t}{\left( 1-
t\right) ^{
2}}}$ is $\left( 3-\sqrt{5}\right) /2>1/3$.
Analyzing the coefficients shows that
we can choose any
$0<r<\left( 3-\sqrt{5}\right) /2$. For
simplicity, we take $r=\frac{
1}{
3}$ and obtain $q
^{\sigma _{0
}}\left( n\right)
 \leq
3
^{n}$.
We define $D^{\sigma}(n)$ for $n \geq 3$ and $n\neq 5$ by
\begin{equation*}
\left\{
\begin{array}{ll}
2
\log _{9/8}\left(
3
\right)
n,&n\equiv 0\pmod{3},\\
\left( 2n+1\right) \log
_{9/8}\left(
3
\right)
-\log
_{9/8}\left(
n+2
\right)
,&n\equiv 1\pmod{3},\\
\log
_{9/8}\left(
2
\right) +
\left( n+1\right)
\log
_{9/8}\left(
3
\right)
,&n\equiv 2\pmod{3}.
\end{array}
\right.
\end{equation*}
Further, $D^{
\sigma }\left( 5\right) :=
\log _{9/8}\left( 3888
\right)
.$
\begin{corollary} 
Let $n \geq 3$. Let $d >
D^{\sigma}(n)$. Then
\begin{equation}
\frac{ \left(
q^{\sigma _{d
}}\left( n\right)  \right)^2                      }{
q^{\sigma _{d
}}\left( n-1\right)      \,\,
q^{\sigma _{d
}}\left( n+1\right) } 
< 1                 \text{, if and only if }n \equiv 1 \pmod{3}.
\end{equation}

\end{corollary}

\section{Proof of Theorem \ref{hauptsatzneu}}
Let $g_{d}$ be fixed satisfying (\ref{Bedingung}).
To simplify notation, we put $
q_{d}\left( n\right) =
q^{g_{d
}}\left( n\right) $. We have
\[
\frac{1}{1-\sum _{n=1}^{\infty } g_{d}\left( n\right) \,
t^{n}}=1+\sum _{n=1
}^{\infty }\,
\sum _{k\leq n}\sum _{\substack{m_{1},\ldots ,m_{k}\geq 1 \\ m_{1}+\ldots +m_{k}=n}}
\!\!\!g_{d}\left( m_{1}\right) \cdots g_{d}\left( m_{k}\right)
t^{n}=\sum _{n=0}^{\infty }
q_{d}\left( n\right)
t^{n}
.
\]
Therefore,
\begin{equation*}
q_{d}\left( n\right) =\sum _{k\leq n}\sum _{\substack{m_{1},\ldots ,m_{k}\geq 1 \\ m_{1}+\ldots +m_{k}=n}}
g_{d}\left( m_{1}\right) \cdots g_{d}\left( m_{k}\right)
\label{eq:p}
\end{equation*}
for $n\geq 1$. 
\subsection{Two Lemmata}

It is known \cite{HN22a} that:
\begin{lemma}
Let $n\geq 2$. Then
\[
\max _{\substack{m_{1},\ldots ,m_{k}\geq 1 \\ m_{1}+\ldots +m_{k}=n}}m_{1}\cdots m_{k}=\left\{
\begin{array}{ll}
3^{n/3},&n\equiv 0\pmod{3},\\
4\cdot 3^{\left( n-4\right) /3},&n\equiv 1\pmod{3},\\
2\cdot 3^{\left( n-2\right) /3},&n\equiv 2\pmod{3}.
\end{array}
\right.
\]
For $n\geq 6$,
$n
\nequiv 2\pmod{3}$, the second largest products are
\begin{eqnarray*}
8\cdot 3^{\left( n
-6\right)
/3},n&\equiv &0\pmod{3},\\
10\cdot 3^{\left( n-7\right) /3},n&\equiv &1\pmod{3}
\end{eqnarray*}
and $3$ for $n=4$.
\end{lemma}
Further, we provide an extension of a result
from \cite{HN22a,HN22b}. 
\begin{lemma}
\label{abschaetzungenallgemein}For
$n\geq 2$
\begin{eqnarray*}
3^{dn/3}&\leq &q
_{d}\left( n\right) \leq 3^{
d
n/3}
q_{0
}\left( n\right) ,n\equiv 0\pmod{3},\\
\left( 4\cdot 3^{\left( n-4\right) /3}\right) ^{d}
\frac{\left( n-1\right) \left( n+8\right) }{
18}
&\leq &
q_{d
}\left( n\right) \leq \left( 4\cdot 3^{\left( n-4\right) /3}\right) ^{d
}
q_{0
}\left( n\right) ,n\equiv 1\pmod{3},\\
\left( 2\cdot 3^{\left( n-2\right) /3}\right) ^{d}\frac{n+1}{3}&\leq &q
_{d}\left( n\right) \leq \left( 2\cdot 3^{\left( n-2\right) /3}\right) ^{d
}
q_{0
}\left( n\right) ,n\equiv 2\pmod{3}.
\end{eqnarray*}
Additionally, for $n\geq 6$, $n
\nequiv 2\pmod{3}$, we have
\begin{eqnarray*}
q_{d}\left( n\right) &\leq &3^{nd/3}+\left( 8\cdot 3^{n/3-2}\right) ^{d
}
q_{0
}\left( n\right) ,n\equiv 0\pmod{3},\\
q_{d}\left( n\right) &\leq &\left( 4\cdot 3^{\left( n-4\right) /3}\right) ^{d}
\frac{\left( n-1\right) \left( n+8\right) }{
18}+
\left( 10\cdot 3^{\left( n-7\right) /3}\right) ^{d
}
q_{0
}\left( n\right) ,n\equiv 1\pmod{3},
\end{eqnarray*}
and $q_{d}\left( 4\right)
\leq
2\cdot 4^{d}+3^{d}
q_{0
}\left( 4\right) $.
\end{lemma}

\begin{proof}
Since
$
m_{1}\cdots m_{k}
\leq \max _{\substack{m_{1},\ldots ,m_{k}\geq 1 \\ m_{1}+\ldots +m_{k}=n}}m_{1}\cdots m_{k}$,
the upper bounds should be obvious as
$g_{d}\left( n\right) \leq n^{d}+\left( g_{
0}\left( n\right) -1
\right) \left( n-1\right) ^{d
}\leq g_{0
}\left( n\right) n^{d
}$.
For the lower
bounds, we obtain
\[
\sum _{k\leq n}\sum _{\substack{m_{1},\ldots ,m_{k}\geq 1 \\ m_{1}+\ldots +m_{k}=n}}\left( m_{1}\cdots m_{k}\right) ^{d}\geq S\left( n\right) \max _{k\leq n}\max _{\substack{m_{1},\ldots ,m_{k}\geq 1 \\ m_{1}+\ldots +m_{k}=n}}\left( m_{1}\cdots m_{k}\right) ^{d}
\]
where $S\left( n\right) $ is the number of
$m_{1}+\ldots +m_{k}=n$, which yield the maximal
product. Therefore,
\[
S\left( n\right) =\left\{
\begin{array}{ll}
1,&n\equiv 0\pmod{3},\\
\frac{n-1}{3}+\binom{\left( n+2\right) /3}{2}=\frac{\left( n-1\right) \left( n+8\right) }{18},&n\equiv 1\pmod{3},\\
\frac{n+1}{3
},&n\equiv 2\pmod{3}.
\end{array}
\right.
\]

For the refined upper bounds, we consider
$g_{d}\left( n\right) \leq n^{d}+\left( g_{
0}\left( n\right) -1
\right) \left( n-1\right) ^{d
}$.
Then
$g_{d}\left( m_{1}\right) \cdots g_{d}\left( m_{k}\right) \leq g_{
0}\left( m_{1}\right) \cdots g_{0
}\left( m_{k}\right) \left( m_{1}\cdots m_{k}\right) ^{d
}$
and for the maximal values
\begin{eqnarray*}
&&g_{d}\left( m_{1}\right) \cdots g_{d}\left( m_{k}\right) \\
&\leq &\left( m_{1}^{d}+\left( g_{
0}\left( m_{1}\right) -
1
\right) \left( m_{1}-1\right) ^{d
}\right) \cdots \left( m_{k}^{d}+\left( g_{
0}\left( m_{k}\right) -1
\right) \left( m_{k}-1\right) ^{d
}\right) .
\end{eqnarray*}
Therefore,
\[
q_{d}\left( n\right)
\leq \left( \max _{k\leq n}\max _{\substack{m_{1},\ldots ,m_{k}\geq 1 \\ m_{1}+\ldots +m_{k}=n}}m_{1}\cdots m_{k}\right) ^{d}S\left( n\right) +\left( Z\left( n\right) \right) ^{d
}
q_{0
}\left( n\right) 
\]
where $Z\left( n\right) $ is the second largest product $m_{1}\cdots m_{k}$ of
all $m_{1}+\ldots +m_{k}=n$.
\end{proof}

\subsection{Proof of Theorem \ref{hauptsatzneu}}
We consider the cases
$n\equiv 0,1,2 \pmod{3}$ and $n=5$ separately.

\subsubsection{Let $n\equiv 0\pmod{3}$}
Then
\begin{eqnarray*}
\frac{\left( q
_{d}\left( n\right) \right) ^{2}}{
q_{d}\left( n-1\right) q
_{d}\left( n+1\right) }&\geq &\frac{3^{2dn/3}}{\left( 2\cdot 3^{\left( n-3\right) /3}\cdot 4\cdot 3^{\left( n-3\right) /3}\right) ^{d
}
q_{0
}\left( n-1\right) q
_{
0}\left( n+1\right) }\\
&\geq &\left( \frac{9}{8}\right) ^{d
}\frac{
1}{
r^{-2n}}\geq 1
\end{eqnarray*}
for
$d\geq
-
2
\log
_{9/8}\left( r\right)
n$.
\subsubsection{Let $n\equiv 1\pmod{3}$}
Then
\begin{eqnarray*}
\frac{\left( q
_{d}\left( n\right) \right) ^{2}}{
q_{d}\left( n-1\right) q
_{d}\left( n+1\right) }&\leq &\frac{\left( \left( 4\cdot 3^{\left( n-4\right) /3}\right) ^{d
}q
_{0
}\left( n\right) \right) ^{2}}{3^{\left( n-1\right) d/3}\cdot \left( 2\cdot 3^{\left( n-1\right) /3}\right) ^{d}\frac{n+2}{3}}\\
&\leq &\left( \frac{8}{9}\right) ^{d
}\frac{3r^{-2n}}{
n+2
}<1
\end{eqnarray*}
for
$d>
\log
_{9/8}\left( 3\right) -2n\log
_{9/8}\left( r\right) -\log
_{9/8}\left(
n+2
\right)
$.
\subsubsection{Let $n\equiv 2\pmod{3}$ and $n\neq 5$}
Then
\begin{equation} \label{eq:logkonkav2mod3}
\frac{\left( q
_{d}\left( n\right) \right) ^{2}}{
q_{d}\left( n-1\right) q
_{d}\left( n+1\right) } \geq 
\frac{\left( \left( 2\cdot 3^{\left( n-2\right) /3}\right) ^{d} \frac{n+1}{3}\right) ^{2}}{ A\left(
d,n\right) \, \, B\left( 
d,n\right) },
\end{equation}
where
\begin{eqnarray*}
A\left(
d,n
\right) & = &   \left( 4\cdot 3^{\left( n-5\right) /3}\right) ^{d}
\frac{\left( n-2\right) \left( n+7\right) }{
18}+
\left( 10\cdot 3^{\left( n-8\right) /3}\right) ^{d
}
q_{0
}\left( n-1\right) ,                    \\
B\left(
d,n
\right) & =&    3^{\left( n+1\right) d/3}+\left( 8\cdot 3^{\left( n-5\right) /3}\right) ^{d
}
q_{0
}\left( n+1\right)                       .
\end{eqnarray*}
Then
as a lower bound
for the
expression on the right hand side of (\ref{eq:logkonkav2mod3})
we obtain the following:
\begin{eqnarray*}
&
&\frac{\left( \frac{n+1
}{3}\right) ^{2}}{\left( \frac{\left( n-2\right) \left( n+7\right) }{
18}+
\left( \frac{5}{6}\right) ^{d
}r^{1-n}
\right) \left( 1+\left( \frac{8}{9}\right) ^{d
}r^{-n-1}
\right) }\\
&\geq &\frac{\left(
\frac{n+1}{3}\right) ^{2}}{\left( \frac{\left( n-2\right) \left( n+7\right) }{
18}+
\frac{4}{9}\right) \frac{3}{2}}\geq \frac{
4\left( n+1\right) ^{2}}{
3\left(
n+6
\right) \left( n
-1
\right) }=1+\frac{\left( n-\frac{7}{2}\right) ^{2}+\frac{39}{4}}{3\left( n+6\right) \left( n-1\right) }
\geq 1
\end{eqnarray*}
for
\begin{eqnarray*}
d&\geq &
\max \left\{
\log _{6/5}\left( 9/
4\right)
-\left( n-1\right) \log
_{6/5}\left( r\right)
,
\log _{9/8}\left( 2\right) -
\left( n+1\right)
\log
_{9/8}\left( r\right)
\right\} \\
&=&
\log _{9/8}\left(
2\right)
-\left( n+
1\right) \log
_{9/8}\left( r\right)
\end{eqnarray*}
for $0<r
\leq 1
$ as
$0\leq -
\log _{6/5}\left( r\right)
\leq -
\log _{9/8}\left( r\right)
$
and
$
\log _{6/5}\left( 9/4\right) 
<
\log _{9/8}\left( 2\right)
$.

\subsubsection{Let $n=5$}

We have
\begin{eqnarray*}
\frac{\left( q
_{d}\left( 5\right) \right) ^{2}}{
q_{d}\left(
4\right) q
_{d}\left( 6
\right) }&\geq
&\frac{
4\cdot 3
6^{d}
}{\left( 2\cdot 4^{d}+3^{d
}q
_{0
}\left( 4\right) \right) \left( 9^{d}+8^{d
}
q_{0
}\left( 6\right) \right) }\\
&\geq &\frac{4\cdot 36^{d}}{2\cdot 36^{d}+
4\cdot 32
^{d
}q
_{0
}\left( 4\right) q
_{0
}\left( 6\right) }\geq 1
\end{eqnarray*}
for
$d\geq
\log _{9/8}\left(
2q_{0
}\left( 4\right)
q_{
0}\left( 6\right)
\right)
$.

\section{Final Remarks}
Let us
examine $\left\{ q
^{g_{d}}\left( n\right) \right\} $. There are no exceptions for $g_d=\sigma _{d}$
or $\psi_d$
for $n=2$, since
\begin{equation*}
\left( q
^{g_{d}}\left( 2\right) \right)^2 - q
^{g_{d}}\left( 1\right) \, q
^{g_{d}}\left( 3\right) = \left( g_{d}\left( 2\right) \right) ^2 - g_d(3).
\end{equation*}

\begin{Challenge} We consider the
exponential case for $d=0$ and $\sigma_0$. We expect $E
^{p^{\sigma _{0}}}$ to be finite. Moreover,
numerical experiments (tested up to $n=25
00$) suggest that 
\begin{equation*}
E^{
p^{\sigma _{0}}}=\left\{ 2k+1\middle| 0\leq k\leq 371\right\} \setminus \left\{ 717,723,729,735,741\right\} .
\end{equation*}
\end{Challenge}

\begin{Challenge}
We consider the geometric case. We have $E^{
q^{\psi _{0}}}= E^{
q^{\psi _{1}}}=\left\{ 1\right\} $, since
$q^{\psi_0}(n) = 2^{n-1}$ and $q^{\psi_1}(n)=F_{2n}$. 
For $E^{
q^{\psi _{2}}}$, we expect infinitely many exceptions and non-exceptions.
\end{Challenge}

\begin{Challenge}[Geometric case]
Let $\sigma_d$ for $0\leq d \leq 4$ be given. Then all the odd numbers $n$ up to
$
10^{4}
$ are exceptions. Note that for $d=5$
some even numbers also appear as exceptions. For example,
\begin{equation*}
\left( q^{\sigma_5}(10) \right)^2 - q^{\sigma _{5}}\left( 9\right) \,q^{\sigma_5}(11) <0.
\end{equation*}
Nevertheless, it seems that the set of
exceptions for each $\sigma_d$ is infinite.
\end{Challenge}
%
%{\bf Acknowledgments.}

%%%%%%%%%%%%%%%%%%%%%%
%%%%%%%%%%%%%%%%%%%%%%

\begin{thebibliography}{BHMPW22}
\bibitem[AHK18]{AHK18} K. Adiprasito, J. Huh, and E. Katz: \emph{Hodge theory for combinatorial
geometries.} Annals of Mathematics \textbf{188} (2018), 381--452.

\bibitem[An98]{An98} G. E. Andrews: \emph{The theory of partitions.} Cambridge Univ.\ Press, Cambridge (1998).

\bibitem[BHMPW22]{BHMPW22}
T. Braden, J. Huh, J. P. Matherne, N. Proudfoot, B. Wang:
\emph{Singular Hodge theory for combinatorial geometries.\/} arXiv:2010.06088v3.

\bibitem[Br89]{Br89} F. Brenti: \emph{Unimodal, log-concave and P\'{o}lya frequency sequences in combinatorics.\/} 
Mem.\ Am.\ Math.\ Soc.\
\textbf{413} (1989).

\bibitem[DP15]{DP15} S. DeSalvo, I. Pak: \emph{Log-concavity of the partition function.}
Ramanujan J.
\textbf{38} (2015), 61--73. 

\bibitem[HN22a]{HN22a} B. Heim, M. Neuhauser: 
\emph{Log-concavity of infinite product generating functions.\/}
\newblock
Res.\ Number Theory \textbf{8} No.\ 3 (2022), Paper No.\ 53, 14 pp.


\bibitem[HN22b]{HN22b} B. Heim, M. Neuhauser: 
\emph{Tur\'an
inequalities for
infinite
product
generating
functions.\/}
\newblock
arXiv:2207.09409v1  [math.CO]  19 Jul 2022.


\bibitem[HNT22]{HNT22} B. Heim, M. Neuhauser, R. Tr\"oger: \emph{Inequalities for plane partitions.}
\newblock
Annals of Combinatorics (2022), 22 pp. doi:10.1007/s00026-022-00604-4.


\bibitem[HZ21]{HZ21} L. Hong, S. Zhang: \emph{Towards Heim and Neuhauser's unimodality conjecture on the Nekrasov--Okounkov polynomials.} Res. Number Theory \textbf{7}
No.\ 1 (2021), Paper No.\ 17, 11 pp.

\bibitem[Hu12]{Hu12}
J. Huh: \emph{Milnor numbers of projective hypersurfaces and the chromatic polynomial
of graphs.} Journal of the American Mathematical Society \textbf{25} (2012), 907--927.


\bibitem[Ka22]{Ka22} G. Kalai: \emph{The work of June Huh.\/} Proc.\ Int.\ Cong.\ Math.\ 2022. Vol.\
{1}.
Preliminary version.

\bibitem[Kr16]{Kr16} C. Krattenthaler: \emph{Plane partitions in the work of Richard Stanley and his school.\/} In:
P. Hersh, T. Lam, P. Pylyavskyy, V. Reiner (eds.)
The mathematical legacy of Richard Stanley. Amer.\ Math.\
Soc., R.~I.\ (2016), 246--277.


\bibitem[Ni78]{Ni78} J.-L. Nicolas: \emph{Sur les entiers $N$ pour 
lesquels il y a beaucoup des groupes ab{\'{e}}liens d'ordre $N$.}
Ann. Inst. Fourier \textbf{28} No. 4 (1978), 1--16.



\bibitem[OR85]{OR85} A. M. Odlyzko, L. B. Richmond: \emph{On the unimodality of high convolutions of discrete distributions.} Ann. Probab. \textbf{13}
No.\ 1 (1985), 299--306.




\bibitem[OPR22]{OPR22} K. Ono, S. Pujahari, L. Rolen: \emph{Tur\'an
inequalities for the plane partition function.}
Advances in Mathematics
\textbf{409} Part B
(2022). doi:10.1016/j.aim.2022.108692.

\bibitem[St89]{St89}
R. Stanley: \emph{Log-concave and unimodal sequences in algebra, combinatorics, and
geometry.} 
In: M. F. Capobianco, M. G. Guan, D. F. Hsu,  F. Tian (eds.) 
Graph theory and its applications: East and West. Proceedings of the first China--USA international conference, held in
Jinan, China, June 9--20, 1986.
Ann.\ New York Acad.\ Sci.\ {\bf 576}, New York Acad. Sci., New York
(1989), 500--535.




\bibitem[St99]{St99} R. Stanley: \emph{Enumerative combinatorics.} Vol. 2. With a foreword by
Gian-Carlo Rota and appendix 1 by Sergey Fomin, Cambridge Studies in Advanced
Mathematics
\textbf{62}, Cambridge University Press, Cambridge
(1999).


\bibitem[Wa22]{Wa22} I. Wagner: \emph{On a new class of Laguerre–P\'olya type 
functions with applications in number theory.} 
Pacific Journal of Mathematics Vol. \textbf{320} No. 1 (2022), 177--192.


\bibitem[Zh22]{Zh22} S. Zhang: 
\emph{Log-concavity in powers of infinite series close to $(1-z)^{-1}$.}
Res.\ Number Theory \textbf{8} No.\ 1
(2022), Paper No.\ 66, 
17 pp. 
\end{thebibliography}
\end{document}